\documentclass[12pt,a4paper]{article}
\usepackage{amsmath}
\usepackage{amssymb}
\usepackage{amssymb}
\usepackage{cases}
\usepackage{bbm}
\usepackage{mathrsfs}
\usepackage{graphicx}
\usepackage{amsfonts}
\usepackage{enumerate}
\usepackage{theorem}
\usepackage[pdfstartview=FitH]{hyperref}

\makeatletter
\def\tank#1{\protected@xdef\@thanks{\@thanks
 \protect\footnotetext[0]{#1}}}
\def\bigfoot{

 \@footnotetext}
\makeatother

\newcommand{\ea}{\end{array}}

\allowdisplaybreaks

\newtheorem{theorem}{Theorem}[section]
\newtheorem{thm}[theorem]{Theorem}

\newtheorem{lem}[theorem]{Lemma}
\newtheorem{proposition}[theorem]{Proposition}

\newtheorem{corollary}[theorem]{Corollary}

\theoremstyle{definition}

\newtheorem{defn}[theorem]{Definition}

\theoremstyle{remark}
\newtheorem{remark}[theorem]{Remark}
\newtheorem{rem}[theorem]{Remark}

\numberwithin{equation}{section}

 \DeclareMathAlphabet{\mathpzc}{OT1}{pzc}{m}{it}

 \newcommand{\M}{\mathcal{M}}

 \newcommand{\E}{\mathbb{E}}            
 \newcommand{\T}{\mathbb{T}}
 \newcommand{\e}{\varepsilon}

 \newcommand{\Ll}{\langle}
 \newcommand{\Rr}{\rangle}

 \newcommand{\DD}{\mathbb{D}}

 \newcommand{\HH}{\mathbb{H}}
 \newcommand{\VV}{\mathbb{V}}

 \newcommand{\N}{\mathbb{N}}
 \newcommand{\R}{\mathbb{R}}
 \newcommand{\Z}{\mathbb{Z}}

 \newcommand{\FF}{\mathcal{F}}
 \newcommand{\PP}{\mathbb{P}}
 \newcommand{\mcl}{\mathcal}

 \newcommand{\beq}{\begin{equation}}
 \newcommand{\nneq}{\end{equation}}
 \newcommand{\Be}{\begin{equation}}
 \newcommand{\Ee}{\end{equation}}
 \newcommand{\Bs}{\begin{split}}
 \newcommand{\Es}{\end{split}}
  \newcommand{\Bes}{\begin{equation*}}
 \newcommand{\Ees}{\end{equation*}}
 \newcommand{\bthm}{\begin{thm}}
 \newcommand{\nthm}{\end{thm}}
 \newcommand{\bprf}{\begin{proof}}
 \newcommand{\nprf}{\end{proof}}
 \newcommand{\blem}{\begin{lem}}
 \newcommand{\nlem}{\end{lem}}
 \newcommand{\bprop}{\begin{proposition}}
 \newcommand{\nprop}{\end{proposition}}
 \newcommand{\bcoro}{\begin{corollary}}
 \newcommand{\ncoro}{\end{corollary}}
 \newcommand{\brem}{\begin{rem}}
 \newcommand{\nrem}{\end{rem}}
 \newcommand{\bdef}{\begin{defn}}
 \newcommand{\ndef}{\end{defn}}

 \newcommand{\eqn}{equation}

  \newcommand{\dif}{{\rm d}}

\def\EE{\mathbb{E}}\def\HH{\mathbb H}
\def\NN{\mathbb N}\def\PP{\mathbb P}
\def\RR{\mathbb{R}}\def\WW{\mathbb W}
\def\SS{\mathbb S}

\def\Om{{\Omega}}

\def\<{\left<}\def\>{\right>}
\def\({\left(}\def\){\right)}

\newenvironment{proof}{\par\noindent{\bf Proof:}}{\hspace*{\fill}$\blacksquare$\par}

\begin{document}
\title{Asymptotics  for stochastic  reaction-diffusion equation driven by  subordinate    Brownian motions}

\footnotesize{\author {Ran Wang $^{1}$\thanks{wangran@ustc.edu.cn},\ \
Lihu Xu $^{2}$\thanks{lihuxu@umac.mo,\ Corresponding author}\\
 {\em $^1$ School of Mathematics and Statistics, }\\
 {\em Wuhan University, Wuhan,  P.R. China.  }\\
  {\em $^2$ Department of Mathematics, Faculty of Science and Technology, }\\
  {\em  University of  Macau, Taipa, Macau.}
}

\maketitle
\begin{minipage}{140mm}
\begin{center}
{\bf Abstract}
\end{center}
We study the ergodicity of stochastic reaction-diffusion equation driven by  subordinate Brownian motions.
   After establishing the strong Feller property and irreducibility of the system,  we prove  the tightness of the solution's law.  These properties imply that this stochastic system admits a unique invariant measure according to Doob's and  Krylov-Bogolyubov's theories.
   Furthermore, we establish a large deviation principle for the occupation measure of  this system by a hyper-exponential recurrence criterion.  It is well known that S(P)DEs driven by $\alpha$-stable type noises do not satisfy Freidlin-Wentzell type large deviation, our result gives an example that
strong dissipation overcomes heavy tailed noises to produce a Donsker-Varadhan type large deviation as time tends to infinity.
\end{minipage}

\vspace{4mm}

\medskip
\noindent
{\bf Keywords}: Stochastic reaction-diffusion equation;  Subordinate  Brownian motions;   Large deviation principle (LDP);  Occupation measure.

\medskip
\noindent
{\bf Mathematics Subject Classification (2000)}: \ {60F10, 60H15,  60J75}.


\section{Introduction}
 Consider a stochastic reaction-diffusion equation driven by subordinate  Brownian motion on torus $\mathbb T:=\mathbb R/\mathbb Z$ as follows:
\begin{\eqn} \label{e:MaiSPDE}
\dif X-\partial_{\xi}^2X\dif t-(X-X^3)\dif t= Q_{\beta} \dif L_t,
\end{\eqn}
where $X:[0,+\infty)\times \mathbb T\times\Om\rightarrow\mathbb R$ and $L_t$ is a  subordinate  Brownian motion. More details about this equation will be given in the next section. Sometimes the equation \eqref{e:MaiSPDE} is also called stochastic Allen-Cahn equation or real Ginzburg-Landau equation.
Recently,  the study of invariant measures and the long time behavior of stochastic partial differential equations (SPDEs) driven by $\alpha$-stable type noises has been  extensively studied,
we refer to \cite{Do08, DXi11, DXZ09,  FuXi09, Mas} and the literatures therein.

\vskip0.3cm
In this paper, we firstly study  the ergodicity of stochastic reaction-diffusion equation driven by  subordinate Brownian motions, showing that the system \eqref{e:MaiSPDE} admits a unique invariant probability measure $\pi$. To do this, we need to prove the system is  strong Feller and irreducible. Those two properties imply the uniqueness of the invariant measure according to Doob's theory (see \cite{Doob}).  To establish the strong Feller property, we truncate the nonlinearity and apply a gradient established in \cite{DXZ14JSP} or \cite{Zh}.   To establish the irreducibility, we  need to prove the irreducibility of the stochastic evolution and then apply a control problem result in \cite{WXX}. Unlike the case of SPDEs driven by cylindrical $\alpha$-stable noises,  the components of the noise are not independent, the approach in the proof of the irreducibility   is  very different from that in our previous paper \cite{WXX}.

\vskip0.3cm
 Another topic is the large deviation principle (LDP) about the occupation measure. Let $\mathcal L_t$ be the occupation measure of the system \eqref{e:MaiSPDE}  given by
\begin{equation}\label{e:occupation1}
\mathcal L_t(A):=\frac1t\int_0^t\delta_{X_s}(A)\dif s \ \ \  \ \text{ for any measurable set } A,
\end{equation}
where $\delta_a$ is the Dirac measure at $a$.  By the uniqueness of invariant measure  (see \cite{DPZ96}), we know that the occupation measure $ \mathcal L_t$ converges to the invariant measure $\pi$. In this paper, we also study  the LDP for the occupation measure $\mathcal L_t$.  The LDP for empirical measures is one of the strongest ergodicity results
for the long time behavior of Markov processes. It has been one of the classical research
topics in probability since the pioneering work of Donsker and Varadhan \cite{DV}. Refer to the books \cite{DZ,DS}.
Based on the hyper-exponential recurrence criterion developed by Wu \cite{Wu01}, we  prove that the occupation measure $\mathcal L_t$ obeys an LDP under $\tau$-topology.  As a consequence, we  can obtain the exact rate of exponential ergodicity.

\vskip0.3cm
For  stochastic partial differential equations,
the problems of  LDP have been extensively studied in recent years. Most of them, however, are concentrated on the small noise LDP of Freidlin-Wentzell type, which provide estimates for the probability that stochastic systems converge to their deterministic part as noises tend to zero.  But there are only very few papers on  the LDP of Donsker-Varadhan tpye for  large time, which  estimate   the probability  of the occupation measures'  deviation from invariant measure. Gourcy \cite{Gou1, Gou2}  established the LDP for occupation measures of stochastic Burgers and Navier-Stokes equations by the means of the hyper-exponential recurrence. Jak\u{s}i\`c et al. \cite{JNPS1} established the LDP for occupation measures of SPDE with smooth random perturbations by Kifer's LDP criterion \cite{Kif}. Jak\u{s}i\`c et al. \cite{JNPS2} also gave the large deviations estimates for dissipative PDEs with rough noise by the hyper-exponential recurrence criterion.
In \cite{WXX2}, using  the hyper-exponential recurrence criterion,  an LDP for the occupation measure  is derived for  a class of non-linear monotone stochastic partial differential equations, such as   stochastic $p$-Laplace equation,  stochastic porous medium equation and stochastic fast-diffusion equation.
\vskip0.3cm

\vskip0.3cm

The paper is organized as follows. In Section 2, we  give a brief review of some known results about  the   stochastic reaction-diffusion equations, and present the main result of this paper. In Sections 3 and 4, we prove the strong Feller property  and the irreducibility of the system separately. In Section 5, we first  recall the hyper-exponential criterion about the LDP for Markov processes, and then verify this condition by establishing some uniform estimates which also imply the tightness of the solution.

\vskip0.3cm

Throughout this paper, $C_p$ is a positive constant depending on some parameter $p$, and $C$ is a constant depending on no  specific parameter (except $\alpha, \beta$), whose value  may be different from line to line by convention.

\section{The model and the results}

Let $\T= \R/\Z$ be equipped with the usual Riemannian metric, and let $\dif \xi$
denote the Lebesgue measure on $\T$. For any $p\ge1$, let
$$
L^p(\T;\R):=\left\{x: \T\rightarrow\R; \|x\|_{L^p}:=\left(\int_\T |x(\xi)|^p \dif\xi\right)^{\frac1p}<\infty\right\}.
 $$
Denote
$$\HH:=\bigg\{x\in L^2(\T; \R); \int_\T x(\xi) \dif\xi =0\bigg\}.$$
$\HH$ is a real separable Hilbert space with inner product
$$\Ll x,y \Rr_{\HH}:=\int_\T x(\xi)y(\xi) \dif\xi,\ \ \ \ \ \forall \ x, y \in \HH.$$
Write $\|x\|_{\HH}:= \left(\langle x,x\rangle_{\HH}\right)^{\frac12}.$

Let $\Delta$ be the Laplace operator on $\HH$. Then $A:=-\Delta$ is a positive self-adjoint operator on $\HH$ with the discrete spectral. More precisely, there exist  an orthogonal basis $\left\{e_k; e_k=e^{i2 \pi k\xi}, \ k \in \Z_*\right\}$ with   $\Z_*:=\Z \setminus \{0\}$,   and a sequence of real numbers  $\left\{\lambda_k=4 \pi^2 |k|^2; \ k \in \Z_*\right \}$ such that $A e_k=\lambda_k e_k$.

 For any $\theta\ge 0$,  let $ \HH_{\theta}$ be the domain of the fractional operator $A^{\frac{\theta}{2}}$, i.e.,
 $$
 \HH_{\theta}:=\left\{\sum_{k\in \Z_*}\lambda_k^{-\frac{\theta}{2}}a_k\cdot e_k:(a_k)_{k\in \Z_*}\subset\RR, \sum_{k\in \Z_*}a_k^2<+\infty \right\},
  $$
 with the inner product
 $$
 \langle u,v\rangle_{\theta}:=\langle A^{\frac{\theta}{2}}u,A^{\frac{\theta}{2}}u\rangle_{\HH}=\sum_{k\in \Z_*} \lambda_k^{\theta}\langle u, e_k\rangle_{\HH}\cdot\langle v, e_k\rangle_{\HH},
 $$
 and with the norm
$\|u\|_{\theta}:=\langle u,u\rangle_{\theta}^{\frac12}=\|A^{\frac{\theta}{2}} u\|_{\HH}.
 $
Clearly, $ \HH_{\theta}$ is densely and compactly embedded in $\HH$. Particularly, let
$$
\VV:=\HH_{1}\ \ \text{and}\ \  \|x\|_{\VV}:=\|x\|_{1}.
$$

\vskip0.3cm
Let $\{W_t^k,t\ge0\}_{k \in \Z_*}$ be a sequence of independent standard one-dimensional Brownian motion on some filtered probability space $(\Omega, \mathcal F,(\mathcal F_t)_{t\ge0},\PP)$. The cylindrical Brownian motion on $\HH$ is defined by
$$
W_t:=\sum_{k\in \Z_*} W_t^k\cdot e_k.
$$
For $\alpha\in (0,2)$, let $S_t$ be an independent $\alpha/2$-stable subordinator, i.e., an increasing one dimensional L\'evy process with Laplace transform
$$
\EE\left[e^{-\eta S_t} \right]=e^{-t|\eta|^{\alpha/2}}, \ \ \ \eta>0.
$$
Then
$
L_t:=W_{S_t}
$
defines a subordinated cylindrical Brownian motion  on $\HH$. Refer to \cite{Ap09, sato}.

For a sequence of bounded real numbers $\beta=(\beta_{k})_{k\in\NN}$, let us define
 $$
 Q_{\beta}:\HH\rightarrow\HH \  \text{ such that }    Q_{\beta}u:=\sum_{k\in \Z_*} \beta_k\langle u,e_k\rangle_{\HH}\cdot e_k\ \ \  \text{ for } u\in \HH.
$$
\vskip0.3cm
We shall rewrite the system \eqref{e:MaiSPDE} into  the following abstract form:
\begin{equation} \label{e:XEqn}
\begin{cases}
\dif X_t+ A X_t\dif t=N(X_t) \dif t+ Q_{\beta}\dif L_t, \\
X_0=x,
\end{cases}
\end{equation}
where
\begin{itemize}
\item[(i)] the nonlinear term $N$ is defined by
\begin{equation*} \label{e:NonlinearB}
N(u)= u-u^3, \ \ \ \ \ u \in \HH;
\end{equation*}
\item[(ii)] $\{L_t\}_{t\ge0}$ is a subordinated cylindrical Brownian motion  on $\HH$ with $\alpha\in (1,2)$, and the intensity  $Q_{\beta}$ satisfies that for some $\delta>0$ and $\frac 32<\theta'\le \theta <2$,
    $$
    \delta \lambda_k^{-\frac{\theta}{2}}\le |\beta_k|\le \delta^{-1}\lambda_k^{-\frac{\theta'}{2}}, \ \ \forall k\in \Z_*.
  $$
    \end{itemize}

\vskip0.3cm

\bdef
\label{d:MSoln}
We say that a predictable $\HH$-valued stochastic process $X=(X_t^x)$ is a mild solution to Eq. \eqref{e:XEqn}, if for any $t\ge0, x\in \HH$, it holds ($\mathbb P$-a.s.):
\begin{equation}\label{e: mild solution}
X^x_t(\omega)=e^{-At} x+\int_0^t e^{-A(t-s)} N(X_s^x(\omega))\dif s+\int_0^t e^{-A(t-s)} Q_{\beta} \dif L_s(\omega).
\end{equation}
\ndef

By Lemma \ref{lem Z} in the next section,  using the similar approach  as in the proof of \cite[Theorem 2.2]{Xu13}, we can    easily obtain that    Eq. \eqref{e:XEqn} admits a unique mild solution $X_{\cdot}(\omega)\in \DD([0,\infty);\HH)\cap \DD((0,\infty);\VV)$. Moreover, $X$ is a Markov process.

\vskip0.3cm

Our first main result is the following theorem about the ergodicity of solution.
\begin{thm}  Assume that $\alpha \in (1,2)$.
Then the  Markov process  $X$ is strong Feller and irreducible  in $\HH$ for any $t>0$, and $X$ admits a unique invariant measure.
\end{thm}
\begin{proof}
We shall prove the the strong Feller property and irreducibility in Section 3 and Section 4.  By the well-known Doob's Theorem (see \cite{DPZ96}),  we know  that  $X$ admits at most one unique invariant probability measure. According to the Krylov-Bogolyubov's theorem (See \cite{DPZ96}), if the family of the law $\{X_t;t\ge1\}$ is tight, then there exists an invariant probability measure for \eqref{e:XEqn}.
The   tightness for  $\{X_t;t\ge1\}$  follows from Theorem \ref{l: X}.

The proof is complete.
\end{proof}

\vskip0.3cm

Recall that $\mathcal L_t$ defined by
\begin{equation}\label{e:occupation}
\mathcal{L}_t(A):=\frac1t\int_0^t\delta_{X_s}(A)\dif s \ \ \  \ \text{ for any measurable set } A,
\end{equation}
where $\delta_a$ is the Dirac measure at $a\in \HH$. Then $\mathcal L_t$ is in  $\mathcal M_1(\HH)$, the space of  probability measures on $\HH$. On  $\mathcal M_1(\HH)$, let  $\sigma(\mathcal M_1(\HH), \mathcal B_b (\HH))$ be the $\tau$-topology  of converence against measurable and bounded functions  which is much stronger than the usual weak convergence topology $\sigma(\mathcal M_1(\HH), C_b(\HH))$, where
  $C_b(\HH)$ is the space of all bounded continuous  functions on $\HH$.  See \cite{DV} or \cite[Section 6.2]{DZ}.
\vskip0.3cm

  Our second main result is about the LDP for occupation time $
\mathcal{L}_t$, whose proof will be given in the last section.
\begin{theorem}\label{thm main} Assume that $\alpha \in (1,2)$.
Then the family $\PP_{\nu}(\mathcal L_T\in \cdot)$ as $T\rightarrow +\infty$ satisfies the LDP with respect to the $\tau$-topology, with speed $T$ and rate function $J$ defined by \eqref{rate func} below, uniformly for any initial measure $\nu$ in $\mathcal M_1(\HH)$.  More precisely, the following three properties hold:
\begin{itemize}
  \item[(a1)]  for any $a\ge0$, $\{\mu\in \mathcal M_1(\HH); J(\mu)\le a \}$ is compact in  $(\mathcal M_1(\HH),\tau)$;
  \item[(a2)] (the lower bound) for any  open set $G$ in $(\mathcal M_1(\HH),\tau)$,
   $$
   \liminf_{T\rightarrow \infty}\frac1T\log\inf_{\nu\in\mathcal M_1(\HH)}\mathbb P_{\nu}(\mathcal L_T\in G)\ge -\inf_G J;
   $$
  \item[(a3)](the upper bound) for any  closed set $F$ in  $(\mathcal M_1(\HH),\tau)$,
   $$
   \limsup_{T\rightarrow \infty}\frac1T\log\sup_{\nu\in\mathcal M_1(\HH)}\mathbb P_{\nu}(\mathcal L_T\in F)\le -\inf_F J.
   $$
\end{itemize}
\end{theorem}

\begin{rem} For every $f:\HH\rightarrow \R$ measurable and bounded, as $\nu\rightarrow\int_{\HH}f\dif \nu$ is continuous w.r.t. the $\tau$-topology, then by the contraction principle (\cite[Theorem 4.2.1]{DZ}), $$\mathbb P_{\nu}\left(\frac{1}{T}\int_0^T f(X_s)\dif s\in \cdot\right)$$
satisfies the LDP on $\R$ uniformly for any initial measure $\nu$ in $\mathcal M_1(\HH)$,  with the rate function given by
$$
J^f(r)=\inf\left\{J(\mu)<+\infty|\mu\in\mathcal M_1(\HH)\  \text{and } \int f\dif\mu=r \right\},\ \ \forall r\in\R.
$$
\end{rem}

\section{Strong Feller property}

 \subsection{Some useful estimates}

We shall often use the following inequalities (see \cite{Xu13}):
\beq \label{e:PoiInq}
\|A^{\sigma_1} x\|_{\HH} \le  C_{\sigma_1, \sigma_2} \|A^{\sigma_2} x\|_{\HH}, \ \ \ \ \ \ \forall \ \sigma_1 \le \sigma_2, x \in \HH;
\nneq
\beq \label{e:eAEst}
\|A^{\sigma} e^{-At}\|_{\HH} \le C_\sigma t^{-\sigma}, \ \ \ \ \ \forall \ \sigma>0, t>0;
\nneq

\beq \label{e:L4}
\|x\|^4_{L^4}\le \|x\|_{\VV}^2\|x\|_{\HH}^2, \ \ \ \ \  \forall  \ x\in \VV;
\nneq
\beq \label{e:NInnPro}
\langle x, N(x)\rangle_{\HH} \le \frac 14, \ \ \ \ \forall \ x \in \HH;
\nneq
\beq \label{e:NVEst}
\|N(x)\|_{\VV} \leq C (\|x\|_{\VV}+\|x\|^3_{\VV}), \ \ \ \ \forall \ x \in \VV;
\nneq
\beq \label{e:Nx-yV}
\|N(x)-N(y)\|_{\VV} \leq C (1+\|x\|^2_{\VV}+\|y\|^2_{\VV})\cdot\|x-y\|_{\VV}, \ \ \ \ \forall \ x, y \in {\VV};
\nneq

\beq \label{e:NxyHEst1}
\|N(x)-N(y)\|_{\HH} \le C(1+\|A^{\frac 14}x\|^2_{\HH}+\|A^{\frac 14}y\|^2_{\HH})\cdot\|x-y\|_{\HH}, \ \ \ \forall \ x,y \in  {\HH}.
\nneq
For all $\sigma \ge \frac 16$,
\beq \label{e:NxyHEst}
\|N(x)-N(y)\|_{\HH} \le C(1+\|A^{\sigma}x\|^2_{\HH}+\|A^{\sigma}y\|^2_{\HH})\cdot\|A^{\sigma}(x-y)\|_{\HH}, \ \ \ \forall \ x,y \in  {\HH};
\nneq
\beq \label{e:NHEst}
\|N(x)\|_{\HH} \le C(1+\|A^{\sigma} x\|^3_{\HH}), \ \ \ \forall \ x \in  {\HH}.
\nneq
\vskip0.3cm

Let us now consider the following stochastic convolution:
 \beq\label{e:Z}
 Z_t:=\int_0^t e^{-(t-s)A}Q_{\beta}\dif L_s=\sum_{k\in \Z_*}\int_0^t e^{-(t-s)\lambda_k}\beta_k\dif W_{S_s}^k\cdot e_k.
 \nneq  The   estimate about $Z_t$ will play an important role in  next sections (cf. \cite{PeZa07, PZ11}).

 \blem\label{lem Z}\cite[Lemma 2.3]{DXZ14JSP} Suppose that for some $\gamma\in \RR$,
$$
 K_{\gamma}:=\sum_{k\in \Z_*}\lambda_k^\gamma |\beta_k|^2<+\infty.
$$
 Then for any $p\in(0,\alpha)$ and $T>0$,
 \beq\label{e: convolution 1}
 \sup_{t\in [0,T]}\EE\left[\|Z_t\|_{\gamma+1}^p \right]\le C_{\alpha, p}K_{\gamma}^{\frac p2} T^{\frac {p}{\alpha}-\frac{p}{2}};
 \nneq
  for any $\theta<\gamma$,
  \beq\label{e: convolution 2}
 \EE\left[\sup_{t\in [0,T]} \|Z_t\|_{\theta}^p \right]\le C_{\alpha, p}K_{\gamma}^{\frac p2} T^{\frac {p}{\alpha}}\left(1+T^{\frac{\gamma-\theta}{2}}\right);
 \nneq
 for any $\varepsilon>0$,
   \beq\label{e: convolution 3}
 \PP\left(\sup_{t\in [0,T]} \|Z_t\|_{\theta}\le \varepsilon  \right)>0.
 \nneq
 Moreover, $t\mapsto Z_t$ is almost surely c\`adl\`ag in $\HH_{\theta}$.

 \nlem

\subsection{Strong Feller property}

For any $f \in \mathcal B_b(\HH)$, $t \ge 0$ and $x \in \HH$, define
$$P_t f(x):=\E[f(X^x_t)].$$

\vskip0.3cm
The main result of this section is
\bthm \label{t:StrFelH}
$(P_t)_{t \ge 0}$,  as a semigroup on $\mathcal B_b(\HH)$, is strong Feller.
\nthm
\noindent To prove  Theorem \ref{t:StrFelH}, thanks to a standard argument (see \cite[p. 943]{Xu13} for example),  we only need to prove that  the  following lemma.
\blem\label{t:StrFel}
$(P_t)_{t \ge 0}$, as a semigroup on $\mathcal B_b(\VV)$, is strong Feller.
\nlem

\begin{proof} The proof is inspired by the proof of Theorem \cite[Theorem 6.2]{Xu13}.
 Let $T_0>0$ be arbitrary, it suffices to show that for all $t \in (0,T_0]$, $x \in \VV$ and $f \in \mathcal B_b(\VV)$,
\beq \label{e:PtfyxLim}
\lim_{\|y-x\|_{\VV} \rightarrow 0}P_t f(y)=P_t f(x).
\nneq
Without loss of generality, we
assume $\|f\|_\infty:=\sup_{x\in \VV}|f(x)|=1$.
We divide the proof into three steps.

 {\bf Step 1.}
Since  the nonlinearity $N$ is not bounded and Lipschitz continuous, we need to use a
 truncation technique. Consider the equation
with truncated nonlinearity as follows:
\beq \label{e:TruEqn}
dX^{\rho}_t+A X^{\rho}_t \dif t =N^{\rho}(X^{\rho}_t)\dif t+Q_{\beta}\dif L_t, \ \ \ X^\rho_0=x \in \VV,
\nneq
where $\rho>0$, $N^{\rho}(x):=N(x) \chi( \|x\|_{\VV}/\rho)$ for all $x \in \VV$ and
$\chi: \R \rightarrow [0,1]$ is a smooth function such that
$$\chi(z)=1 \ \ \ {\rm for} \ |z| \le 1, \ \ \ \ \chi(z)=0 \ \ \ {\rm for} \ |z| \ge 2.$$
By \eqref{e:NVEst}, for all $x \in \VV$,
\beq \label{e:BouB}
\|N^{\rho} (x)\|_{\VV} \le C (\|x\|_{\VV}^3+\|x\|_{\VV}) \chi\left( \|x\|_{\VV}/\rho\right) \le C(\rho^3+\rho).
\nneq  It follows from   \eqref{e:Nx-yV} that
\beq
\begin{split}
\|N^{\rho} (x)-N^{\rho}(y)\|_{\VV} \le  C(1+\rho^2)\cdot \|x-y\|_{\VV}.
\end{split}
\nneq
Hence,
Eq. \eqref{e:TruEqn} admits a unique Markov  solution $X^\rho_{.} \in \DD([0,\infty);\VV)$.

 By Theorem 3.1 in \cite{DXZ14JSP} (choosing $\sigma=\gamma=1$ and $\gamma'=0$ there), we have for any $0<t\le T_0$, and  $x,y\in \VV$,
 \beq\label{eq DXZ lem 3.1}
 |\E[f(X^{\rho,x}_t)]-\E[f(X^{\rho,y}_t)]| \le   Ct^{-\frac{1}{\alpha}-\frac{\theta-1}{2}} \|f\|_{\infty}\cdot \|x-y\|_{\VV}.
 \nneq

\vskip 0.3cm

{\bf Step 2.} Define
$$K_{T_0}(\omega):= \sup_{0 \le t \le T_0}\|Z_t(\omega)\|_{\VV}, \ \ \ \omega \in \Omega.$$
By Lemma \ref{lem Z} and Markov inequality, we have
\beq \label{e:KOmeSma}
\PP(K_{T_0}>\rho/2) \le   C(\alpha, T_0)/{\rho},
\nneq
where $C(\alpha, T_0)$ is some constant depending on $\alpha$ and $T_0$.

Choose $\rho$ so large that $\|x\|_{\VV}\le \sqrt \rho<  \rho/2-1$ and define
$$G:=\{K_{T_0}\le \rho/2\}.$$
For all $\omega \in \Omega$, define $Y_t(\omega):= X_t(\omega)-Z_t(\omega)$, then
 \begin{equation*}
\dif Y_t+A Y_t\dif t=N(Y_t+Z_t)\dif t, \ \ \ \ Y_0=x\in \VV.
 \end{equation*}
By (ii) of Lemma 4.1 in \cite{Xu13}, there exists some
$0<t_0 \le T_0$ depending on $\rho$ such that for all
$\omega \in G$,
$$
\sup_{0 \le t \le t_0} \|Y^x_t(\omega)\|_{\VV} \le 1+\|x\|_{\VV}  \le 1+\sqrt \rho<\rho/2.
$$
Then
\begin{align}\label{eq S1}
 \PP\left(\sup_{0 \le t \le t_0} \|X^x_t\|_{\VV} \ge \rho\right) &\le \PP\left(\sup_{0 \le t \le t_0} \|Y^x_t\|_{\VV}
+\sup_{0 \le t \le T_0} \|Z_t\|_{\VV} \ge \rho\right)  \notag \\
& \le \PP\left(K_{T_0}>\rho/2\right)+\PP\left(\sup_{0 \le t \le t_0} \|Y^x_t\|_{\VV}>\rho/2, G\right)\notag\\
& = \PP\left(K_{T_0}>\rho/2\right).
\end{align}
The above inequality,  together with   \eqref{e:KOmeSma} and \eqref{eq S1},  implies that
\beq\label{eq X sup}
\PP\left(\sup_{0 \le t \le t_0} \|X^x_t\|_{\VV} \ge \rho\right) \le \PP(K_{T_0}>\rho/2)\le C(\alpha, T_0)/\rho.
\nneq

\vskip0.3cm
{\bf Step 3}.
Define the stopping time
 \begin{align*}\label{stopping time}
 \tau_x:=\inf\{t>0; \|X_t^x\|_{\VV}\ge \rho \}.
 \end{align*}
By \eqref{eq X sup}, we obtain that for all $t \in [0,t_0]$,
\beq \label{e:TauEst}
\PP_x(\tau_{x} \le t)=\PP\left(\sup_{0 \le s \le t} \|X^x_s\|_{\VV} \ge \rho\right) \le  C(\alpha, t)/\rho.
\nneq
Since Eqs. \eqref{e:XEqn} and   \eqref{e:TruEqn} both have a unique mild solution, for all $t \in [0, \tau_{x})$,
we have
\beq \label{e:StrWeaUni}
X^{\rho,x}_t=X^x_t \ \ a.s..
\nneq

Let $y \in \VV$ be such that $\|x-y\|_{\VV} \le 1$ and choose $\rho>0$ be sufficiently
large so that $\max\{\|x\|_{\VV}, \|y\|_{\VV} \} \le \sqrt{\rho}$.
For any  $t \in (0,t_0]$, it holds that
\begin{align}\label{eq I}
|P_t f(x)-P_t f(y)|=|\E [f(X^x_t)]-\E [f(X^y_t)]|=I_1+I_2+I_3,
\end{align}
where
\begin{align*}
& I_1:=|\E[f(X^x_t)1_{[\tau_{x}> t]}]-\E[f(X^y_t)1_{[\tau_{y} > t]}]|, \\
& I_2:=|\E[f(X^x_t)1_{[\tau_{x} \le t]}]|,\\
&I_3:=|\E[f(X^y_t)1_{[\tau_{y} \le t]}]|.
\end{align*}

It follows from  \eqref{e:TauEst} that
\beq\label{eq I23}
I_2 \le \frac{C\|f\|_{\infty}}{\rho}, \ \ I_3 \le \frac{C\|f\|_{\infty}}{\rho}.
\nneq

\noindent It remains to estimate $I_1$. It follows from  \eqref{eq DXZ lem 3.1}, \eqref{e:TauEst} and  \eqref{e:StrWeaUni}   that
\begin{align}\label{eq I1}
I_1& =\left|\E[f(X^{\rho,x}_t)1_{[\tau_{x}>t]}]-\E[f(X^{\rho,y}_t)1_{[\tau_{y} > t]}]\right|\notag \\
&\le \left|\E[f(X^{\rho,x}_t)]-\E[f(X^{\rho,y}_t)]\right|+\left|\E[f(X^{\rho,x}_t)1_{[\tau_{x} \le t]}]|+|\E[f(X^{\rho,y}_t)1_{[\tau_{y} \le t]}]\right|\notag\\
& \le  Ct^{-\frac{1}{\alpha}-\frac{\theta-1}{2}}\|f\|_{\infty}\cdot  \|x-y\|_{\VV}+2C \|f\|_{\infty}/\rho.
\end{align}
 For all $\e>0, t \in (0,t_0]$,
choosing $$\rho \ge \max\left\{\frac{12C \|f\|_{\infty}}{\e}, 2\|x\|_{\VV}^2+2\right\}, \ \ \delta=\frac{\e }{2  C} t^{\frac{1}{\alpha}+\frac{\theta-1}{2}},$$
 by Eqs. \eqref{eq I}, \eqref{eq I23} and \eqref{eq I1}, we obtain that for all   $\|x-y\|_{\VV} \le \delta$,
$$|P_t f(x)-P_t f(y)|<\e.$$

As $t_0<t\le T_0$, it follows from  the Markov property and the strong Feller property above that
$$
P_t f(y)-P_t f(x)= P_{t_0}[P_{t-t_0}f](y)-P_{t_0}[P_{t-t_0}f](x)\rightarrow 0,
$$
as $\|y-x\|_{\VV}\rightarrow 0$.

The proof is complete.
\end{proof}

\section{Irreducibility}

The main result of this part is the irreducibility of the stochastic dynamics.
\begin{thm}\label{thm irred}
Assume that $\alpha \in (1,2)$.
For any initial value $x\in \HH$, the  Markov process $X=\{X_t^x\}_{t\ge0,x\in \HH}$ to  Eq. \eqref{e:XEqn} is irreducible in $\HH$.
\end{thm}

\begin{remark}\label{rmk} By the well-known Doob's Theorem (see \cite{DPZ96}), the strong Feller property and the irreducibility imply that  $X$ admits at most one unique invariant probability measure.   \end{remark}

\subsection{Irreducibility of stochastic convolution}

Let  $\SS$ be the space of all  increasing and c\`adl\`ag functions from $(0,\infty)$ to $(0,\infty)$ with $\lim_{s\rightarrow 0^+}l_s=0$, which is endowed with the Skorohod metric  and the probability measure $\mu_{\SS}$ so that the coordinate process
$ S_t(l):=l_t$ is an $\alpha/2$-stable subordinator.

Consider the following product probability space
$$
(\Omega, \mathcal F, \PP):=(\WW\times\SS, \mathcal B(\WW)\times\mathcal B(\SS),\mu_{\WW}\times\mu_{\SS})
$$
and define
$$
L_t(w,l):=w_{l_t}.
$$
We shall use the following two natural filtration associated with the L\' evy process $L_t$ and   the Brownian motion $W_t$:
$$
\mathcal F_t:=\sigma\{L_s(w,l);s\le t\},\  \  \  \ \mathcal F_t^{\WW}:=\sigma\{W_s(w); s\le t\},
$$
and denote by   $\EE^{\SS}$  and $\EE^{\WW}$ the partial integrations with respect to $S$ and $W$, respectively.

\vskip0.3cm

For any $l\in\SS$, let $Z_t^l$ solve the following equation:
\beq\label{e Xl}
\dif Z^l_t+ A Z^l_t\dif t = Q_{\beta}\dif W_{l_t},\
Z^l_0=0.
\nneq
 It is well known that
$$
Z_t^l=\int_0^t e^{-A(t-s)}Q_{\beta} \dif W_{l_s}=\sum_{k \in \Z_{*}} \beta_k z_k(t)\cdot e_k,
$$
where $$z_{k}(t)=\int_0^t e^{-\lambda_k(t-s)}
 \dif W^k_{l_s}.$$
Notice that for any fixed $l\in\SS$, $\{z_{k}\}_{k\in \Z_{*}}$ are independent by the independence of $\{W^k_{\cdot}\}_{k\in \Z_{*}}$.

We claim that for any $\gamma \in (0,\theta'-1)$ with $\theta'$ defined above Definition \ref{d:MSoln},
\ \ \
\Be  \label{e:EGam}
\E^{\WW}\left[ \sup_{0 \le t \le T} \|A^\gamma Z^l_t\|_{\HH}\right] \le C_{\gamma,\theta'} \sqrt{l_T}.
\Ee
Indeed, upper to a standard finite dimension approximation argument, using integration by parts we get
\beq\label{eq integ}
Z^l_t=Q_{\beta} W_{l_t}-\int_0^t A e^{-A(t-s)} Q_{\beta} W_{l_s} \dif s,
\nneq
which clearly implies
\ \ \
\Bes
\begin{split}
\sup_{0 \le t \le T} \|A^\gamma Z^l_t\|_{\HH} & \le \sup_{0 \le t \le T}\|A^\gamma Q_{\beta} W_{l_t}\|_{\HH}+\sup_{0 \le t \le T} \int_0^t \|A^{1+\gamma} e^{-A(t-s)} Q_{\beta} W_{l_s} \|_{\HH} \dif s \\
& \le \sup_{0 \le t \le T}\|A^\gamma Q_{\beta} W_{l_t}\|_{\HH}+\sup_{0 \le t \le T} \int_0^t \|A^{1+\gamma-\gamma'} e^{-A(t-s)}\|\cdot \|A^{\gamma'} Q_{\beta} W_{l_s} \|_{\HH} \dif s \\
& \le \sup_{0 \le t \le T}\|A^\gamma Q_{\beta} W_{l_t}\|_{\HH}+C \sup_{0 \le t \le T}\|A^{\gamma'} Q_{\beta} W_{l_t} \|_{\HH}\cdot\sup_{0 \le t \le T} \int_0^t (t-s)^{1+\gamma-\gamma'} \dif s \\
& \le \sup_{0 \le t \le l_T}\|A^\gamma Q_{\beta} W_{t}\|_{\HH}+C T^{\gamma'-\gamma} \sup_{0 \le t \le l_T}\|A^{\gamma'} Q_{\beta} W_{t} \|_{\HH},
\end{split}
\Ees
where $\gamma' \in (\gamma, \theta'-1)$. Hence, by the martingale inequality we get \eqref{e:EGam}.
\vskip0.3cm

 The following lemma  is concerned with  the support of  the  distribution of $\big(\{Z_t\}_{0 \le t \le T},Z_T\big)$.
\begin{lem} \label{l:SupZt}
For  any $T>0, 0< p<\infty$, the random variable $\left(\{Z_t\}_{0 \le t \le T},Z_T\right)$ has a full support in $L^p([0,T];\VV)\times \VV$. More precisely, for any $\phi\in L^p([0,T];\VV), a\in \VV, \e>0$,
\beq
\PP\left(\int_0^T \|Z_t-\phi_t\|_{\VV}^p \dif t+\|Z_T-a\|_{\VV}<\e\right)>0.
\nneq
\end{lem}

\begin{proof} The proof is divided into several steps.

{\bf Step 1.} (Finite dimensional projection) For  any $N \in \N$, let $\HH_N$ be the Hilbert space spanned by $\{e_k\}_{ 1\le |k| \le N}$, and let
$\pi_N: \HH \rightarrow \HH_N$ be the orthogonal projection.  Notice that $\pi_N$ is also an orthogonal projection in $\VV$.
Define
$$\pi^N:=I-\pi_N, \ \ \ \ \HH^N:=\pi^N \HH.$$
Then for  any given $l\in\SS$,  $\pi_N Z^l$ and $\pi^N Z^l$ are independent. Thus,  for any $\phi\in L^p([0,T];\VV)$ and  $a\in \VV$, we have
\begin{align*}
&  \PP\left(\int_0^T \|Z_t-\phi_t\|_{\VV}^p \dif t+\|Z_T-a\|_{\VV}<\e\right) \\
=& \EE^{\SS}\left(\PP^{\WW}\left(\int_0^T \|Z^l_t-\phi_t\|_{\VV}^p \dif t+\|Z^l_T-a\|_{\VV}<\e\right)\bigg|_{l=S}\right) \\
\ge &\EE^{\SS}\bigg(\PP^{\WW}\bigg(\int_0^T \|\pi_N(Z^l_t-\phi_t)\|_{\VV}^p \dif t+\|\pi_N(Z^l_T-a)\|_{\VV}<\frac{\e}{2^{p+1}}, \\
& \ \ \ \ \ \ \ \ \  \ \ \int_0^T \|\pi^N(Z^l_t-\phi_t)\|_{\VV}^p \dif t+\|\pi^N(Z^l_T-a)\|_{\VV}<\frac{\e}{2^{p+1}}\bigg) \bigg|_{l=S}\bigg) \\
=&\EE^{\SS}\bigg(\PP^{\WW}\bigg(\int_0^T \|\pi_N(Z^l_t-\phi_t)\|_{\VV}^p \dif t+\|\pi_N(Z^l_T-a)\|_{\VV}<\frac{\e}{2^{p+1}}\bigg)\bigg|_{l=S} \\
&\ \ \ \ \ \ \ \    \times \PP^{\WW}\left(\int_0^T \|\pi^N(Z^l_t-\phi_t)\|_{\VV}^p \dif t+\|\pi^N(Z^l_T-a)\|_{\VV}<\frac{\e}{2^{p+1}}\right)\bigg|_{l=S}\bigg).
\end{align*}
For any $\gamma \in (\frac 12, \theta'-1)$,
by the spectral gap inequality, Chebyshev inequality and \eqref{e:EGam}, we have for any $\eta>0$
\begin{align*}
\PP^{\WW}\left(\sup_{0 \le t \le T} \|\pi^N Z^l_t\|_{\VV}>{\eta}\right)
 \le &\PP^{\WW}\left(\sup_{0 \le t \le T} \|\pi^N A^{\gamma} Z^l_t\|_{\HH}>{\eta} \lambda^{\gamma-\frac 12}_N\right) \notag \\
 \le &\PP^{\WW}\left(\sup_{0 \le t \le T} \|A^{\gamma} Z^l_t\|_{\HH}>{\eta} \lambda^{\gamma-\frac 12}_N\right)  \notag \\
 \le & C_{\theta',\gamma} \sqrt{l_T} {\eta}^{-1} \lambda_N^{\frac12-\gamma}. \notag
\end{align*}
Hence,
\begin{align*}
&  \PP\left(\int_0^T \|Z_t-\phi_t\|_{\VV}^p \dif t+\|Z_T-a\|_{\VV}<\e\right) \\
\ge &\EE^{\SS}\bigg[\PP^{\WW}\bigg(\int_0^T \|\pi_N(Z^l_t-\phi_t)\|_{\VV}^p \dif t+\|\pi_N(Z^l_T-a)\|_{\VV}<\frac{\e}{2^{p+1}}\bigg)\bigg|_{l=S} \times \left(1- C_{\theta',\gamma} \sqrt{l_T} 2^{p+1} {\e}^{-1} \lambda_N^{\frac12-\gamma}\right)\bigg|_{l=S} \bigg] \\
\ge & \EE^{\SS}\bigg[\PP^{\WW}\bigg(\int_0^T \|\pi_N(Z^l_t-\phi_t)\|_{\VV}^p \dif t+\|\pi_N(Z^l_T-a)\|_{\VV}<\frac{\e}{2^{p+1}}\bigg)\bigg|_{l=S} \times \left(1- C_{\theta',\gamma} T 2^{p+1} {\e}^{-1} \lambda_N^{\frac12-\gamma}\right), S_T \le T^2\bigg] \\
\ge & \frac 12\EE^{\SS}\bigg[\PP^{\WW}\bigg(\int_0^T \|\pi_N(Z^l_t-\phi_t)\|_{\VV}^p \dif t+\|\pi_N(Z^l_T-a)\|_{\VV}<\frac{\e}{2^{p+1}}\bigg)\bigg|_{l=S}, S_T \le T^2\bigg],
\end{align*}
as $N$ is sufficiently large.

As long as we prove that for any $\e>0$
\Be \label{e:Est0}
\EE^{\SS}\bigg[\PP^{\WW}\bigg(\int_0^T \|\pi_N(Z^l_t-\phi_t)\|_{\VV}^p \dif t+\|\pi_N(Z^l_T-a)\|_{\VV}<\e\bigg)\bigg|_{l=S}, S_T \le T^2\bigg]>0,
\Ee
the proof is complete.
\vskip0.3cm

{\bf Step 2}. It remains to prove \eqref{e:Est0}. Since $\pi_N Z^l_t=\sum_{|i| \le N} z_i(t) e_i$ with $\{z_i(t)\}_i $ being independent stochastic processes, it suffices to prove \eqref{e:Est0} for  one  dimensional case, i.e., for any $\phi \in L^p([0,T];\R)$, $a \in \R$ and $\e>0$,
\ \ \
\Be \label{e:Est21}
\E^{\SS}\left[\PP^{\WW}\left(\int_0^T |z(t)-\phi(t)|^p dt+|z(T)-a|<\e\right)\bigg|_{l=S}, S_T \le T^2\right]>0,
\Ee
where $z(t)=\int_0^t e^{-\lambda (t-s)} \dif w_{l_s}$ with $\lambda>0$ and $w_{t}$ being a one dimensional Brownian motion. To prove
\eqref{e:Est21}, we only need to show that
 \Be \label{e:Est2}
\E^{\SS}\left[\PP^{\WW}\left(\int_0^T \left|\int_0^t e^{\lambda s}\dif w_{l_s}-e^{\lambda t} \phi(t)\right|^p dt+\left|\int_0^T e^{\lambda s}\dif w_{l_s}-e^{\lambda T}a\right|<\e\right)\bigg|_{l=S}, S_T \le T^2\right]>0,
\Ee

Since the  simple function space is dense in $L^p([0,T];\RR)$, without loss of generality, we assume that $e^{\lambda t} \phi(t)$ is a simple function vanished at $t=0$ and having the form:
\ \ \
\beq
e^{\lambda t}\phi(t)=\sum_{j=0}^{m-1} a_j 1_{[t_j,t_{j+1})}(t)+a_m 1_{\{t_m\}}(t)
\nneq
where $0=t_0<t_1<...<t_{m}=T$ and $a_0=0$, $a_1 \in \R,...,a_{m-1} \in \R$ and $a_m=e^{\lambda T} a$.

Define $M:=\sup_{1 \le k \le m} |a_k|$ and
\ \ \
\Bes
\begin{split}
\Delta_{t_0,...,t_m,\sigma, \delta}=\big\{&l \in \SS: 0=\tau_0<\tau_1<\tau_2<...<\tau_{m-1}<\tau_m<T \ {\rm such \ that}
\ l_T \le T^2\\
&\tau_i \in (t_i-\sigma,t_i+\sigma), l_{\tau_{i}}-l_{{\tau_i}-}\in (\delta,2\delta)
\ {\rm for} \ 0 \le i \le m\big\}.
\end{split}
\Ees
It is easy to see for all $\sigma>0,\delta>0$,
\ \ \
\Be \label{e:PDelta}
\PP^{\SS}(\Delta_{t_0,...,t_m,\sigma,\delta})>0.
\Ee
Choosing $\sigma<\frac12 \min_{1\le k \le m} (t_k-t_{k-1})\le \frac{T}{2}$, we immediately get
\Be \label{e partition} t_k-\sigma<\tau_k<t_k+\sigma<t_{k+1}-\sigma<\tau_{k+1}<t_{k+1}+\sigma, \ \ \ \ \ k=0,1,...,m-1. \Ee
Denote
$$\Delta a_j:=a_{j}-a_{j-1}, \ \Delta w_{j}:=w_{l_{\tau_j}}-w_{l_{\tau_j-}}, \ I_j:=\sup_{\tau_j \le t<\tau_{j+1}-} \left|\int_{\tau_{j-1}}^{\tau_j-} e^{\lambda s} \dif w_{l_s}\right|^p.$$
 Notice that $e^{\lambda T}a=a_m=\sum_{j=1}^m \Delta a_j$, it is easy to check
\begin{align}\label{e:T1}
\left|\int_0^T e^{\lambda s}\dif w_{l_s}-e^{\lambda T}a\right|&=\left|\sum_{j=1}^{m} \left(\int_{\tau_{j-1}}^{\tau_j-} e^{\lambda s} \dif w_{l_s}+e^{\lambda \tau_j}
\Delta w_j\right)+\int_{\tau_m}^T e^{\lambda s} dw_{l_s}-a_m\right|\notag\\
& \le \sum_{j=1}^m \left|e^{\lambda \tau_j}  \Delta w_{j}-\Delta a_j\right|+\sum_{j=1}^m I^{1/p}_j+\left|\int_{\tau_m}^T e^{\lambda s} \dif w_{l_s}\right|,
\end{align}
 and
\begin{align}\label{e: T2}
&\int_0^T \left|\int_0^t e^{\lambda s}\dif w_{l_s}-e^{\lambda t} \phi(t)\right|^p dt\notag\\
=&\sum_{k=1}^m \int_{\tau_{k-1}}^{\tau_k} \left|\int_0^t e^{\lambda s}\dif w_{l_s}-e^{\lambda t}\phi(t)\right|^p dt+\int_{\tau_m}^T |e^{\lambda t} z(t)-e^{\lambda t} \phi(t)|^p dt.
\end{align}
Since $\phi(t)=0$ for all $t \in [0,t_1)$, we have
\  \ \ \
\Be
\begin{split}
\int_{0}^{\tau_1} \left|\int_0^t e^{\lambda s}\dif w_{l_s}-e^{\lambda t}\phi(t)\right|^p dt=\int_{0}^{\tau_1} \left|\int_0^t e^{\lambda s} dw_{l_s}\right|^p dt \le 2^{p-1}T (I_1+e^{\lambda p \tau_1} \left|\Delta w_1\right|^p).
\end{split}
\Ee
For $t \in [\tau_k,\tau_{k+1})$ with $k=1,...,m-1$, we have
\ \ \
\Be \label{e:IntExp}
\int_0^t e^{\lambda s} \dif w_{l_s}=\sum_{j=1}^{k} \left[\int_{\tau_{j-1}}^{\tau_j-} e^{\lambda s} \dif w_{l_s}+e^{\lambda \tau_j}
\Delta w_j\right]+\int_{\tau_k}^t e^{\lambda s} dw_{l_s}.
\Ee
Let us now compare $[\tau_k, \tau_{k+1})$ with $[t_k, t_{k+1})$, it is easy to see that
\Be\label{e: J}
\int_{\tau_{k}}^{\tau_{k+1}} \left|\int_0^t e^{\lambda s}\dif w_{l_s}-e^{\lambda t} \phi(t)\right|^p dt \le 3^{p-1}( J_{k1}+J_{k2}+J_{k3}),
\Ee
where
\Bes
\begin{split}
& J_{k1}:=\left(\int_{\tau_{k}}^{t_k} \left|\int_0^t e^{\lambda s}\dif w_{l_s}-e^{\lambda t} \phi(t)\right|^p dt\right)1_{\{\tau_k<t_k\}},  \\
&J_{k2}:=\int_{t_{k}}^{t_{k+1} \wedge \tau_{k+1}} \left|\int_0^t e^{\lambda s}\dif w_{l_s}-e^{\lambda t} \phi(t)\right|^p dt, \\
&J_{k3}:=\left(\int_{t_{k+1}}^{\tau_{k+1}} \left|\int_0^t e^{\lambda s}\dif w_{l_s}-e^{\lambda t} \phi(t)\right|^p dt\right)1_{\{t_{k+1}<\tau_{k+1}\}}.
\end{split}
\Ees

 By the easy relation $a_{k-1}=\sum_{j=1}^k  \Delta a_j-\Delta a_{k}$, \eqref{e partition} and \eqref{e:IntExp}, we further obtain
 \begin{align*}
J_{k1}&=\left(\int_{\tau_k}^{t_k} \left|\sum_{j=1}^{k} \int_{\tau_{j-1}}^{\tau_j-} e^{\lambda s} \dif w_{l_s}+\sum_{j=1}^{k}\left(e^{\lambda \tau_j}
\Delta w_j-\Delta a_j\right)+\int_{\tau_{k}}^t e^{\lambda s} dw_{l_s}+\Delta a_k\right|^p dt\right) 1_{\{\tau_k<t_k\}} \notag \\
& \le \sigma (2k+2)^{p-1} \left[\sum_{j=1}^{k}I_j+\sum_{j=1}^{k}\left|e^{\lambda \tau_j} \Delta w_j-\Delta a_j\right|^p+\sup_{\tau_k\le t<\tau_{k+1}}\left|\int_{\tau_{k}}^{t} e^{\lambda s} dw_{l_s}\right|^p+|\Delta a_k|^p\right] \notag \\
&=\sigma (2k+2)^{p-1} \left[\sum_{j=1}^{k+1}I_j+\sum_{j=1}^{k}\left|e^{\lambda \tau_j} \Delta w_j-\Delta a_j\right|^p+|\Delta a_k|^p\right].
\end{align*}

For $J_{k2}$, by the similar argument, we have
\begin{align*}
J_{k2} & \le (2k+1)^{p-1} (t_{k+1}-t_k) \left[\sum_{j=1}^{k}I_j+\sum_{j=1}^{k}\left|e^{\lambda \tau_j} \Delta w_j-\Delta a_j\right|^p+\sup_{\tau_k\le t<\tau_{k+1}}\left|\int_{\tau_{k}}^{t} e^{\lambda s} dw_{l_s}\right|^p\right]\notag \\
&=(2k+1)^{p-1} (t_{k+1}-t_k) \left[\sum_{j=1}^{k+1}I_j+\sum_{j=1}^{k}\left|e^{\lambda \tau_j} \Delta w_j-\Delta a_j\right|^p\right].
\end{align*}
Moreover, as $t_{k+1}<\tau_{k+1}$ we similarly have
\Be
\begin{split}
J_{k3} \le \sigma (2k+2)^{p-1} \left[\sum_{j=1}^{k+1} I_j+\sum_{j=1}^{k}\left|e^{\lambda \tau_j} \Delta w_j-\Delta a_j\right|^p+|\Delta a_{k+1}|^p\right].
\end{split}
\Ee
Therefore, for $k=1,...,m-1$,
\begin{align}\label{eq t1}
&\int_{\tau_{k}}^{\tau_{k+1}} \left|\int_0^t e^{\lambda s}\dif w_{l_s}-e^{\lambda t}\phi(t)\right|^p dt\notag\\
\le& 2T (2k+2)^{p-1} \left[\sum_{j=1}^{k+1} I_j+\sum_{j=1}^{k}\left|e^{\lambda \tau_j} \Delta w_j-\Delta a_j\right|^p\right]+2\sigma (2k+2)^{p-1}\left(|\Delta a_k|^p+|\Delta a_{k+1}|^p\right).
\end{align}

Similarly, we have
 \begin{align}\label{eq t2}
&\int_{\tau_m}^T \left|\int_0^t e^{\lambda s}\dif w_{l_s}-e^{\lambda t} \phi(t)\right|^p dt\notag\\
 \le & \sigma(2m+2)^{p-1}
\left[\sum_{j=1}^m I_j+\sum_{j=1}^m \left|e^{\lambda \tau_j} \Delta w_j-\Delta a_j \right|^p+|\Delta a_m|^p+\left|\int_{\tau_m}^T e^{\lambda s}d w_{l_s}\right|^p\right].
\end{align}

Hence, by \eqref{e: T2}, \eqref{e: J}, \eqref{eq t1} and \eqref{eq t2}, we have
\ \ \
\Bes
\begin{split}
&\int_0^T \left|\int_0^t e^{\lambda s}\dif w_{l_s}-e^{\lambda t} \phi(t)\right|^p dt\\
  \le & 6^{p}T( m+1)^{p} \left[\sum_{j=1}^{m} I_j+\sum_{j=1}^{m}\left|e^{\lambda \tau_j} \Delta w_j-\Delta a_j\right|^p\right] \\
&+\sigma 6^{p} (m+1)^{p} \sum_{k=1}^m |\Delta a_k|^p +\sigma 2^{p-1}(m+1)^{p-1} \left|\int_{\tau_m}^T e^{\lambda s}d w_{l_s}\right|^p,
\end{split}
\Ees
this, together with \eqref{e:T1}, immediately gives
\ \ \
\Be
\begin{split}
& \ \ \ \int_0^T \left|\int_0^t e^{\lambda s}\dif w_{l_s}-e^{\lambda t} \phi(t)\right|^p dt+\left|\int_0^T e^{\lambda s}\dif w_{l_s}-e^{\lambda T}a\right| \\
& \le  6^{p}T( m+1)^{p} \left[\sum_{j=1}^{m} (I_j+I_j^{1/p})+\sum_{j=1}^{m}\left(\left|e^{\lambda \tau_j} \Delta w_j-\Delta a_j\right|^p+\left|e^{\lambda \tau_j} \Delta w_j-\Delta a_j\right|\right)\right] \\
& \ \ \ \ +\sigma 6^{p} (m+1)^{p}\sum_{k=1}^m |\Delta a_k|+\left[\sigma   2^{p-1}(m+1)^{p-1} \left|\int_{\tau_m}^T e^{\lambda s}d w_{l_s}\right|^p+\left|\int_{\tau_m}^T e^{\lambda s}d w_{l_s}\right|\right].
\end{split}
\Ee
For any $\e \in (0,1)$, by the easy fact $\max_{1 \le i \le m} |\Delta a_i| \le 2M$, choose $\sigma>0$ sufficiently small,    we have
$$\sigma 6^{p} (m+1)^{p}\sum_{k=1}^m |\Delta a_k|\le \sigma 6^{p} (m+1)^{p}\cdot 2 mM<\frac{\e}{2}.$$
Therefore the event
$$\left\{\int_0^T \left|\int_0^t e^{\lambda s}\dif w_{l_s}-e^{\lambda t} \phi(t)\right|^p dt+|e^{\lambda T} z(T)-e^{\lambda T} a| \le \e\right\}\subset \bigcap_{i=1}^{m} \left(A_i(l) \cap B_i(l)\right),$$
with
 \begin{align*}
& \ A_i(l):=\left\{I_i+I^{1/p}_i \le \frac{\e}{ 8\cdot6^{p}T( m+1)^{p}}\right\}, \\
& \ B_i(l):=\left\{\left|e^{\lambda \tau_i} \Delta w_i-\Delta a_i\right|^p+\left|e^{\lambda \tau_i} \Delta w_i-\Delta a_i\right| \le \frac{\e}{8\cdot6^{p}T( m+1)^{p}}\right\}, \\
& \ C(l):=\left\{   \left|\int_{\tau_m}^T e^{\lambda s}d w_{l_s}\right|^p+\left|\int_{\tau_m}^T e^{\lambda s}d w_{l_s}\right|<\frac{\e} {8\cdot 2^{p-1}(m+1)^{p-1} T}\right\},
\end{align*}
for all $i=1,...,m$.  Given the subordinator $S=l$, $A_1(l), B_1(l), ..., A_{m}(l), B_{m}(l), C(l)$ are independent. By the reflection property of Brownian motion (see Proposition 3.3.7 in \cite{RY}), it is easy to calculate that for all $l \in \Delta_{t_0,...,t_m,\sigma, \delta}, i=1,2,\cdots, m$, we have
\Bes
 \PP^{\WW} \left(A_i(l)\right)>0, \ \ \  \PP^{\WW} \left(B_i(l)\right)>0, \ \ \ \PP^{\WW} \left(C(l)\right)>0,
\Ees
while  the last inequality can also be obtained by  Eq. \eqref{eq integ}. Thus,
\begin{align}\label{eq p}
&\E^{\SS} \left[\PP^{\WW}\left(\int_0^T|z(t)-\phi(t)|^p dt+|z(T)-a|\le \e\right) \bigg|_{l=S}, S \in \Delta_{t_0,...,t_m,\sigma,\delta}\right]\notag\\
\ge & \E^{\SS} \left\{\PP^{\WW}\left[\cap_{i=1}^{m} \left(A_i(l) \cap B_i(l)\right)\cap C(l)\right] \bigg|_{l=S}, S \in \Delta_{t_0,...,t_m,\sigma,\delta}\right\} \notag \\
=&\E^{\SS} \left\{\prod_{i=1}^{m} \PP^{\WW} \left(A_i(l)\right) \PP^{\WW} \left(B_i(l)\right) \PP(C(l))\bigg|_{l=S}, S \in \Delta_{t_0,...,t_m,\sigma, \delta}
\right\}. \notag
\end{align}
Which, together with \eqref{e:PDelta}, immediately implies
\ \ \
\Be
\E^{\SS} \left[\PP^{\WW}\left(\int_0^T|z(t)-\phi(t)|^p dt+|z(T)-a|\le \e\right) \bigg|_{l=S}, S \in \Delta_{t_0,...,t_m,\sigma,\delta}\right]>0.
\Ee
Since $\{S \in \Delta_{t_0,...,t_m,\sigma,\delta}\} \subset \{S_T \le T^2\}$, we immediately get \eqref{e:Est2}, as desired.

The proof is complete.
\end{proof}

\vskip0.3cm
\subsection{Irreducibility   in $\HH$}

Consider the deterministic system in $\HH$,
\begin{equation}\label{e: deterministic}
\partial_t x(t)+Ax(t)=N(x(t))+u(t), \ \ \ x(0)=x_0,
\end{equation}
where $u\in L^2([0,T];\VV)$. By using the similar argument in the proof of Lemma 4.2 in \cite{Xu13},  for every $x(0)=x_0\in \HH, u\in L^2([0,T];\VV)$, the system \eqref{e: deterministic} admits a unique solution $x(\cdot)\in C([0, T];\HH)\cap C((0, T];\VV)$.  Moreover, $\{x(t)\}_{t\in[0,T]}$ has the following form:
\begin{equation}\label{e: solu deter}
x(t)=e^{-At} x_0+\int_0^t e^{-A(t-s)} N(x(s))\dif s+\int_0^t e^{-A(t-s)} u(s)\dif s, \ \ \ \forall \ t\in[0,T].
\end{equation}

 In \cite{WXX},   the following control problem of the deterministic system is  proved.
\begin{lem}\cite[Lemma 3.3]{WXX} \label{l:AppCon}
For any $T>0, \e>0, a \in \VV$, there exists some $u \in L^\infty([0,T];\VV)$ such that the system \eqref{e: deterministic} satisfies that
\begin{equation*}
\|x(T)-a\|_{\VV} < \e.
\end{equation*}
\end{lem}

 Now we  prove Theorem \ref{thm irred}  by following the idea  in \cite[Theorem 5.4]{PZ11}. This approach   has been used in the proof of Theorem  \cite[Theorem 2.3]{WXX}.    For the convenience of  reading, we give the  proof here.

\begin{proof}[Proof of Theorem \ref{thm irred}] Since Eq. \eqref{e:XEqn} admits a unique mild solution $X_{\cdot}\in \DD([0,\infty);\HH)\cap\DD((0,\infty);\VV)$, for any $x_0\in \HH, t>0$,  we have $X_t^{x_0}\in \VV$ a.s.. By the Markov  property of $X$,  for any $a\in \HH, T>0,\e>0$,
\begin{align*}
\PP \left(\|X_T^{x_0}-a\|_{\HH}<\e\right)=&\int_{\VV} \PP \left(\|X_T^{x_0}-a\|_{\HH}<\e| X_t^{x_0}=v\right)\PP(X_t^{x_0}\in \dif v)\\
=&\int_{\VV} \PP \left(\|X_{T-t}^{v}-a\|_{\HH}<\e\right)\PP(X_t^{x_0}\in \dif v).
\end{align*}
 To prove that
\begin{equation*}\label{e: ir}
\PP \left(\|X_T^{x_0}-a\|_{\HH}<\e\right)>0,
\end{equation*}
 it is sufficient to prove that for any $T>0$,
$$
\PP \left(\|X_T^{x_0}-a\|_{\HH}<\e\right)>0 \ \ \ \ \text{for all } x_0\in \VV.
$$
   Next, we prove the theorem under the assumption of the initial value $x_0\in \VV$ in  the following two steps.
\vskip0.3cm
\emph{Step 1}.  For any $a\in \HH,\e>0$, there exists some $\theta>0$ such that $e^{-\theta A}a\in \VV$ and
\begin{equation}\label{e: irre 1}
\|a-e^{-\theta A}a\|_{\HH}\le \frac \e 4.
\end{equation}
For any $T>0$, by Lemma \ref{l:AppCon} and the spectral gap inequality, there exists some
$u \in L^\infty([0,T];\VV)$ such that the system
\begin{equation*}
\dot x+Ax=N(x)+u, \ \ \ x(0)=x_0,
\end{equation*}
satisfies that
\begin{equation}\label{e: irre 2}
\|x(T)-e^{-\theta A}a\|_{\HH} \le\|x(T)-e^{-\theta A}a\|_{\VV}< \frac\e 4.
\end{equation}
Putting  \eqref{e: irre 1} and \eqref{e: irre 2} together, we have
\begin{equation}\label{e: irre 3}
\|x(T)-a\|_{\HH}< \frac\e 2.
\end{equation}

\vskip0.3cm

\emph{Step 2:}
We shall consider the systems \eqref{e:Lin} and \eqref{e:LinS}
as follows:
 \begin{equation}\label{e:Lin}
    \begin{cases}
 \dot z+A z=u, \ \ \ \ z(0)=0, \\
\dot y+A y=N(y+z), \ \ \ \ y(0)=x_0\in \VV,
\end{cases} \end{equation}
and
 \begin{equation}\label{e:LinS}
    \begin{cases}
\dif Z_t+A Z_t \dif t= Q_{\beta}\dif L_t, \ \ \ \ Z_0=0;\\
\dif Y_t+A Y_t\dif t=N(Y_t+Z_t)\dif t, \ \ \ \ Y_0=x_0\in \VV.
 \end{cases} \end{equation}
By the arguments in the proof of Lemma 4.2 in \cite{Xu13}, for any $x_0\in \VV, u\in L^2([0,T];\VV)$, the systems \eqref{e:Lin} and \eqref{e:LinS} admit the unique solutions $(y(\cdot), z(\cdot)) \in  C([0,T];\VV)^2$ and $(Y_{\cdot}, Z_{\cdot}) \in   C([0,T];\VV)\times \DD([0,T];\VV)$, a.s.
 Furthermore, denote
$$
x(t)=y(t)+z(t),\ \ \ X_t=Y_t+Z_t,\ \ \ \ \forall t\ge0.
$$
For any $0 \le t\le T$,
\begin{align*}
 &  \|Y_t-y(t)\|_{\HH}^2 +2\int_{0}^t \|Y_s-y(s)\|^2_{\VV} \dif s \\
=&2 \int_{0}^t \Ll Y_s-y(s), N(X_s)-N(x(s))\Rr_\HH \dif s \\
=& 2\int_{0}^t \|Y_s-y(s)\|^2_\HH \dif s+2 \int_0^t \Ll Y_s-y(s),Z_s-z(s)\Rr_\HH \dif s \\
& -2 \int_{0}^t \left\Ll Y_s-y(s), X^3_s-x^3(s)\right\Rr_\HH \dif s.
 \end{align*}
Let us estimate the third term of the right hand side. Denoting $\Delta Y_s=Y_s-y(s)$ and
$\Delta Z_s=Z_s-z(s)$, we have
\begin{equation*}
\begin{split}
 &\int_{0}^t \left\Ll Y_s-y(s), X^3_s-x^3(s)\right\Rr_\HH \dif s\\
=&  \int_{0}^t \left\Ll \Delta Y_s, [\Delta Y_s+\Delta Z_s+x(s)]^3-x^3(s)\right\Rr_\HH \dif s\\
=&\int_{0}^t \langle \Delta Y_s,[\Delta Y_s+\Delta Z_s]^3+3[\Delta Y_s+\Delta Z_s]^2 x(s)+3[\Delta Y_s+\Delta Z_s] x^2(s) \rangle_\HH \dif s\\
=& \int_{0}^t \langle \Delta Y_s, (\Delta Y_s)^3  +3 (\Delta Y_s)^2 \Delta Z_s +3 \Delta Y_s (\Delta Z_s)^2 +(\Delta Z_s)^3\rangle_\HH \dif s \\
&+3\int_{0}^t \left\langle \Delta Y_s,[(\Delta Y_s)^2+2 \Delta Y_s \Delta Z_s+ (\Delta Z_s)^2] x(s)\right\rangle_\HH\dif s +3\int_{0}^t \langle \Delta Y_s,[\Delta Y_s+\Delta Z_s] x^2(s)\rangle _\HH\dif s.
\end{split}
\end{equation*}
Since $\frac 34  (\Delta Y_s)^4+3 (\Delta Y_s)^3 x(s)+3(\Delta Y_s)^2x^2(s) \ge 0$, from the above relation we have
\begin{align*}
 &\int_{0}^t \left\Ll Y_s-y(s), X^3_s-x^3(s)\right\Rr_\HH \dif s \\
\ge& \int_{0}^t \langle \Delta Y_s,3 (\Delta Y_s)^2 \Delta Z_s+3 \Delta Y_s (\Delta Z_s)^2 +(\Delta Z_s)^3\rangle_\HH \dif s\\
&+3\int_{0}^t \langle \Delta Y_s,[2 \Delta Y_s \Delta Z_s+(\Delta Z_s)^2] x(s) \Rr_\HH \dif s\\
&+3\int_{0}^t \langle \Delta Y_s, \Delta Z_s x^2(s)\rangle_\HH \dif s+\frac14 \int_{0}^t \|\Delta Y_s\|_{L^4}^4\dif s.
\end{align*}
Using the following Young inequalities: for all $y,z\in L^4(\T; \R)$,
\begin{equation}\label{eq Young}
\begin{split}
& |\Ll  y, z \Rr_\HH|=\left|\int_{\mathbb T} y(\xi) z (\xi) \dif \xi\right| \le \frac{\int_{\mathbb T} y^4(\xi) \dif \xi}{80}+ C\int_{\mathbb T} z^{\frac 43}(\xi) \dif \xi, \\
& |\Ll  y^2, z \Rr_\HH|=\left|\int_{\mathbb T} y^2(\xi) z (\xi) \dif \xi\right| \le \frac{\int_{\mathbb T} y^4(\xi) \dif \xi}{80}+C\int_{\mathbb T} z^2(\xi) \dif \xi. \\
& |\Ll  y^3, z  \Rr_\HH|=\left|\int_{\mathbb T} y^3(\xi) z(\xi) \dif \xi\right| \le \frac{ \int_{\mathbb T} y^4(\xi) \dif \xi}{80}+ C\int_{\mathbb T} z^4(\xi) \dif \xi,
\end{split}
\end{equation}
and the H\"older inequality, we further get
\begin{align*}
 &\int_{0}^t \left\Ll Y_s-y(s), X^3_s-x^3(s)\right\Rr_\HH \dif s \\
\ge&  \frac{1}{80}\int_{0}^t \|\Delta Y_s\|_{L^4}^4\dif s-7C\int_{0}^t \|\Delta Z_s\|_{L^4}^4 \dif s\\
&-6C\int_{0}^t\|\Delta Z_s x(s)\|_{L^2}^2\dif s-3C\int_{0}^t\|(\Delta Z_s)^2 x(s)\|_{L^{\frac43}}^{\frac43}\dif s\\
&-3C\int_{0}^t\|\Delta Z_s x^2(s)\|_{L^{\frac43}}^{\frac43}\dif s\\
\ge&  \frac{1}{80}\int_{0}^t \|\Delta Y_s\|_{L^4}^4\dif s-7C\int_{0}^t \|\Delta Z_s\|_{L^4}^4 \dif s\\
&-6C\int_{0}^t\|\Delta Z_s\|_{L^4}^{2}\|x(s)\|_{L^4}^{2}\dif s-3C\int_{0}^t\|\Delta Z_s\|_{L^4}^{\frac83}\|x(s)\|_{L^4}^{\frac43}\dif s\\
&-3C\int_{0}^t\|\Delta Z_s\|_{L^4}^{\frac43}\|x(s)\|_{L^4}^{\frac83}\dif s.
\end{align*}
 Since $x(t)=y(t)+z(t)\in   C([0, T]; \VV)$, by \eqref{e:L4}, there exists a constant $C_T$ such that
 $$
\sup_{s\in [0, T]} \|y(s)+z(s)\|_{L^4}\le  \sup_{s\in [0, T]}\|y(s)+z(s)\|_{\HH}^{\frac12}\cdot\|y(s)+z(s)\|_{\VV}^{\frac12}\le C_T.
 $$
Consequently,  there is some constant $C_T>0$ satisfying that
\begin{equation*}
\begin{split}
& \|Y_t-y(t)\|_{\HH}^2+2\int_{0}^t \|Y_s-y(s)\|^2_{\VV} \dif s \\
\le& 3\int_{0}^t \|Y_s-y(s)\|^2_\HH \dif s+\int_{0}^t \|Z_s-z(s)\|^2_\HH \dif s\\
& +C_T\int_{0}^t\Big( \|Z_s-z(s)\|_{L^4}^4+\|Z_s-z(s)\|_{L^4}^{2} +\|Z_s-z(s)\|_{L^4}^{\frac83} +\|Z_s-z(s)\|_{L^4}^{\frac43} ds\Big) \dif s.
 \end{split}
\end{equation*}
Therefore, by the spectral gap inequality and Gronwall's inequality, we have
\begin{equation}\label{e: Gron}
\begin{split}
\|Y_T-y(T)\|_{\HH}^2 \le  C_T \sum_{i\in \Lambda} \int_{0}^T \|Z_s-z(s)\|_{\VV}^{i }  \dif s,
\end{split}
\end{equation}
where $\Lambda:=\{  4/3, 2, 8/3,4\}$. This inequality, together with Lemma \ref{l:SupZt}, \eqref{e: irre 3}, implies
\begin{equation*} 
\begin{split}
&\PP \left(\|X_T-a\|_{\HH}<\e\right)\\
=&\PP \left(\|Y_T-y(T)+Z_T-z(T)+x(T)-a\|_{\HH}<\e\right) \\
\ge&\PP \left(\|Y_T-y(T)\|_{\HH} \le \e/4, \|Z_T-z(T)\|_{\HH} \le  \e/4, \|x(T)-a\|_{\HH}<\e/2\right) \\
=&\PP \left(\|Y_T-y(T)\|_{\HH} \le \e/4, \|Z_T-z(T)\|_{\HH} \le  \e/4\right)\\
\ge& \PP \left(   \sum_{i\in \Lambda} \int_{0}^T \|Z_s-z(s)\|_{\VV}^{i }  \dif s+\|Z_T-z(T)\|_{\VV} \le C_{T, \e}\right)\\
>&0.
\end{split}
\end{equation*}
The proof is complete.
\end{proof}

\section{LDP for the occupation time}

\subsection{LDP for the occupation time}

In this section, we recall some general results on the LDP for strong Feller and irreducible   Markov processes. We follow \cite{Wu01}.

\vskip0.3cm

Let $E$ be a Polish metric space. Consider a general $E$-valued c\`adl\`ag  Markov process
$$
\left(\Omega, \{\mathcal F_t\}_{t\ge0}, \mathcal F, \{X_t(\omega)\}_{t\ge0}, \{\mathbb P_x\}_{x\in E}\right),
$$
where
\begin{itemize}
  \item  $\Omega=D( [0,+\infty); E)$, which is the space of the c\`adl\`ag functions from $[0,+\infty)$ to $E$ equipped with the  Skorokhod topology; for any $\omega\in \Omega$, $X_t(\omega)=\omega(t)$;
  \item $\FF_t^0=\sigma\{X_s; 0\le s\le t\}$ for any $t\ge 0$ (nature filtration);
  \item $\FF=\sigma\{X_t; t\ge0\}$ and $\PP_x(X_0=x)=1$.
\end{itemize}
 Hence, $\PP_{x}$ is the law of the Markov process with initial state $x\in E$.  For any initial measure $\nu$ on $E$, let $\PP_{\nu}(\dif \omega):=\int_E \PP_x(\dif \omega)\nu(\dif x)$. Its transition probability is denoted by $\{P_t(x, dy)\}_{t\ge0}$.

For all $f\in b\mathcal B(E)$, define
$$
P_tf(x)=\int_E P_t(x, \dif y)f(y)  \ \ \ \text{for all } t\ge0, x\in E.
$$
 $\{P_t\}_{t\ge0}$ is  {\it accessible }to $x \in E$,  if the resolvent $\{\mcl R_{\lambda}\}_{\lambda>0}$ satisfies
$$\mcl R_{\lambda}(y, \mcl U):= \int_0^\infty e^{-\lambda t}P_t(y, \mcl U) \dif t>0 , \ \ \forall \lambda>0$$
for all $y \in E$ and all neighborhoods $\mcl U$ of $x$. Notice that the accessibility of $\{P_t\}_{t\ge0}$ to any $x\in E$ is   the so called {\it topological transitivity} in Wu \cite{Wu01}.

\vskip0,3cm

The empirical measure of level-$3$ (or process level) is given by
$$
R_t:=\frac1t\int_0^t \delta_{\theta_s X}\dif s
$$
where $(\theta_sX)_t=X_{s+t}$ for all $t, s\ge0$ are the shifts on $\Omega$. Thus, $R_t$ is a random element of $\mathcal M_1(\Omega)$, the space of all probability measures on $\Omega$.

The level-$3$ entropy functional of Donsker-Varadhan $H:\mathcal M_1(\Omega)\rightarrow [0,+\infty]$ is defined by
\begin{equation*} \label{e: DV}
H(Q):=\begin{cases}
 \mathbb E^{\bar Q}h_{\mathcal F_1^0}(\bar Q_{w(-\infty,0]};\mathbb P_{w(0)}) & \text{if } Q\in \mathcal M_1^s(\Omega);  \\
+\infty & \text{otherwise},
\end{cases}
\end{equation*}
where
 \begin{itemize}
   \item $\mathcal M_1^s(\Omega)$ is the subspace of $\mathcal M_1(\Omega)$, whose  elements are moreover stationary;
   \item $\bar Q$ is the unique stationary extension of $Q\in \mathcal M_1^s(\Omega)$ to $\bar \Omega:=D(\mathbb R; E)$; $\mathcal F_t^s=\sigma\{X(u); s\le u\le t\},\forall  s,t\in\R, s\le t$;
   \item $\bar Q_{w(-\infty,t]}$ is the regular conditional distribution of $\bar Q$ knowing $\mathcal F_t^{-\infty}$;
   \item $h_{\mathcal G}(\nu;\mu)$ is the usual relative entropy or Kullback information of $\nu$ with respect to $\mu$ restricted to the $\sigma$-field $\mathcal G$, given by
\begin{equation*}
h_{\mathcal G}(\nu;\mu):=\begin{cases}
  \int\frac{\dif\nu}{\dif\mu}|_{\mathcal G} \log\left(\frac{\dif\nu}{\dif\mu}|_{\mathcal G}\right) \dif\mu & \text{ if }  \nu\ll \mu \text{ on } \ \mathcal G;  \\
+\infty & \text{otherwise}.
\end{cases}
\end{equation*}
 \end{itemize}

The level-$2$ entropy functional $J: \mathcal M_1(E)\rightarrow [0, \infty]$ which governs the LDP in our main result is
\begin{equation}\label{rate func}
J(\mu)=\inf\{H(Q)| Q\in \mathcal M_1^s(\Omega) \  \ \text{and } Q_0=\mu\}, \ \ \ \ \forall \mu\in \mathcal M_1(E),
\end{equation}
where $Q_0(\cdot)=Q(X_0\in \cdot)$ is the marginal law at $t=0$.

\subsubsection{The hyper-exponential recurrence criterion}
   Recall the following hyper-exponential recurrence criterion for LDP established by Wu \cite[Theorem 2.1]{Wu01}.

\vskip0.3cm
For  any measurable set $K\in E$, let
\begin{equation}\label{stopping time}
\tau_K:=\inf\{t\ge0 \ \text{ s.t.}\   X_t\in K\},\ \ \ \tau_K^{(1)}:=\inf\{t\ge1\  \text{ s.t.}\ X_t\in K\}.
\end{equation}

\begin{theorem}\cite{Wu01}\label{thm Wu}
Let $\mathcal A\subset \mathcal M_1(E)$ and assume that
\begin{equation*}\label{condition 1}
\{P_t\}_{t\ge0} \text{ is strong Feller and topologically irreducible on  } E.
\end{equation*}
If for any $\lambda>0$, there exists some compact set $K\subset \subset E$, such that
\begin{equation}\label{condition 2}
\sup_{\nu\in\mathcal A}\E^{\nu}e^{\lambda\tau_K}<\infty, \ \  \text{and} \ \ \
\sup_{x\in K}\E^{x}e^{\lambda\tau_K^{(1)}}<\infty.
\end{equation}
 Then the family $\mathbb P_{\nu}(\mathcal L_t\in\cdot)$ satisfies the LDP on $\mathcal M_1(E)$ w.r.t. the $\tau$-topology with the rate function $J$ defined by \eqref{rate func}, and uniformly for initial measures $\nu$ in the subset $\mathcal A$. More precisely, the following three properties hold:
\begin{itemize}
  \item[(a1)] for any $a\ge0$, $\{\mu\in \mathcal M_1(E); J(\mu)\le a \}$ is compact in  $(\mathcal M_1(E),\tau)$;
  \item[(a2)] (the lower bound) for any open set $G$ in $(\mathcal M_1(E), \tau)$,
   $$
   \liminf_{T\rightarrow \infty}\frac1T\log\inf_{\nu\in\mathcal A}\mathbb P_{\nu}(\mathcal L_T\in G)\ge -\inf_G J;
   $$
  \item[(a3)](the upper bound) for any  closed set $F$ in $(\mathcal M_1(E), \tau)$,
   $$
   \limsup_{T\rightarrow \infty}\frac1T\log\sup_{\nu\in\mathcal A}\mathbb P_{\nu}(\mathcal L_T\in F)\le -\inf_F J.
   $$
\end{itemize}

\end{theorem}

\subsection{The proof of  Theorem \ref{thm main} }
In this section, we shall prove  Theorem \ref{thm main} according to Theorem \ref{thm Wu}.

\begin{proof} [Proof of Theorem \ref{thm main}]
\ Let $\{X_t\}_{t\ge0}$ be the solution to Eq. \eqref{e:XEqn} with initial value $x\in \HH$. By Theorems \ref{t:StrFelH} and \ref{thm irred}, we know that $X$ is strong Feller and irreducible in $H$.  According to Theorem \ref{thm Wu}, to prove Theorem \ref{thm main}, we need prove that the hyper-exponential recurrence condition \ref{condition 2} is fulfilled. The verification of this condition  will be given by Theorem  \ref{thm hyper-exp 2} below.
\end{proof}

\vskip0.3cm

Let $Y_t:=X_t-Z_t$. Then $Y_t$ satisfies the following equation:
\begin{equation}\label{e:Y}
\dif Y_t+A Y_t\dif t=N(Y_t+Z_t)\dif t, \ \ \ \ Y_0=x.
\end{equation}

\begin{lem}\label{l:Y}  For all $T>0$, we have
 \begin{equation}  \label{e:YtEstZt}
\sup_{t\in [T/2,T]}\|Y_t\|_{\HH}\le C(T)\left(1+\sup_{0\le t\le T}\|Z_t\|_{\VV} \right),
\end{equation}
 where the constant $C(T)$     does not depend on the initial value $Y_0=x$.
\end{lem}

\begin{proof}
By the chain rule, we obtain that
\begin{equation}\label{e:Y1}
\frac{\dif \|Y_t\|_{\HH}^2}{\dif t}+2\|Y_t\|_{\VV}^2=2\langle Y_t, N(Y_t+Z_t) \rangle.
\end{equation}
Using the  Young inequalities \eqref{eq Young},  H\"older inequality and the elementary inequality $2\sqrt a\le a/b+b$ for all $a,b>0$, we  obtain that there exists a constant $C\ge1$ satisfying that
$$
2\langle Y_t, N(Y_t+Z_t) \rangle\le -\|Y_t\|_{L^4}^4+C(1+\|Z_t\|_{L^4}^4).
$$
This  inequality, together with  Eq. \eqref{e:L4},  Eq. \eqref{e:Y1} and   H\"older inequality, implies that
\begin{equation}\label{e:Y2}
\frac{\dif \|Y_t\|_{\HH}^2}{\dif t}+2\|Y_t\|_{\VV}^2\le -\|Y_t\|_{\HH}^4+C\left(1+\|Z_t\|_{\VV}^4\right).
\end{equation}

For any $t\ge0$, denote $$h(t):= \|Y_t\|_{\HH}^2,\ \  \ K_T:=\sup_{0\le t\le T}\sqrt{C(1+\|Z_t\|_{\VV}^4)}\ge 1.$$
 By Eq. \eqref{e:Y2}, we have
$$
\frac{\dif h(t)}{\dif t} \le -h^2(t)+K_T^2, \ \ \ \forall t\in[0,T],
$$
with the initial value $h(0)=\|x\|_{\HH}^2\ge0$.

By the comparison theorem (e.g., the deterministic case of   \cite[Chapter VI, Theorem 1.1]{IW}),  we obtain that
\begin{equation}\label{e:h g}
h(t)\le g(t), \ \ \ \ \ \ \forall t\in[0,T],
\end{equation}
where the function $g$ solves the following equaiton
\begin{equation}\label{e:g0}
\frac{\dif g(t)}{\dif t}= -g^2(t)+K_T^2, \ \ \ \forall t\in[0,T],
\end{equation}
with the initial value $g(0)=h(0)$. The solution of  Eq. \eqref{e:g0} is
\begin{equation*}
g(t)=K_T+2K_T\left( \frac{g(0)+K_T}{g(0)-K_T}e^{2K_T t}-1\right)^{-1}, \ \  \ \forall t\in[0,T],
\end{equation*}
where  it is understood that  $g(t)\equiv K_T$ when  $g(0)=K_T$.  It is easy to show that
 for any initial value $g(0)$, we have
 $$
 g(t)\le K_T\left(1+2(   e^{T}-1)^{-1}\right), \ \ \ \forall t\in[T/2,T].
  $$
  This inequlity, together with Eq. \eqref{e:h g} and the definition of $K_T$, immediately implies the required estimate \eqref{e:YtEstZt}.

         The proof is complete.
\end{proof}

\vskip0.3cm

\begin{lem}\label{l:Y2}
For all $T\ge1$, $\delta \in (0,1)$ and  $p \in (0,\alpha/4)$, we have
\begin{align*}
\E^{x}\left[\|Y_{T}\|^p_{\delta} \right]\le  C_{\alpha, p}T,
\end{align*}
where the constant $C_{\delta, p}$  does not depend on the initial  value $Y_0=x$ and $T$.
\end{lem}
\begin{proof} Since
$$Y_{T}=e^{-AT/2} Y_{T/2}+\int_{T/2}^T e^{-A(T-s)} N(Y_s+Z_s) \dif s,$$
for any $\delta \in (0,1)$, by the inequalities \eqref{e:eAEst}-\eqref{e:NHEst} and Lemma \ref{l:Y}, there exists a constant $C=C_{T,\delta}$ (whose value may be different from line to line by convention) satisfied  that
\begin{align*}
\|Y_{T}\|_{\delta}
\le & C \|Y_{T/2}\|_{\HH}+C\int_{T/2}^T (T-s)^{-\frac{\delta}{2}} \|N(Y_s+Z_s)\|_{\HH} \dif s\notag \\
  \le& C \|Y_{T/2}\|_{\HH}+C\int_{T/2}^T (T-s)^{-\frac{\delta}{2}} (\|Y_s\|_{\HH}+\|Z_s\|_{\HH}+ \|Y_s^3\|_{\HH}+\|Z_s^3\|_{\HH}) \dif s \notag \\
 \le & C \|Y_{T/2}\|_{\HH}+C\int_{T/2}^T (T-s)^{-\frac{\delta}{2}} (\|Y_s\|_{\HH}+\|Z_s\|_{\VV}+\|Y_s\|_{\VV}\cdot \|Y_s\|^2_{\HH}+\|Z_s\|^3_{\VV}) \dif s \notag\\
  \le & C \left(1+\sup_{0 \le t \le T} \|Z_t\|^3_{\VV}\right)+C\int_{T/2}^T (T-s)^{-\frac{\delta}{2}} \|Y_s\|_{\VV}\cdot \|Y_s\|^2_{\HH} \dif s.
\end{align*}
Next, we estimate  the last term in above inequality: by Eq. \eqref{e:Y2} and Lemma \ref{l:Y} again, we have
\begin{align*}
&\int_{T/2}^T (T-s)^{-\frac{\delta}{2}} \|Y_s\|_{\VV} \|Y_s\|^2_{\HH} \dif s\\
 \le &C\left(1+\sup_{0 \le t \le T}\|Z_t\|^2_{\VV}\right)\int_{T/2}^T (T-s)^{-\frac{\delta}{2}} \|Y_s\|_{\VV} \dif s \\
\le& C\left(1+\sup_{0 \le t \le T}\|Z_t\|^2_{\VV}\right)\left(\int_{T/2}^T (T-s)^{-\delta} \dif s\right)^{\frac 12} \left(\int_{T/2}^T\|Y_s\|^2_{\VV} \dif s\right)^{\frac 12}  \\
\le& C\left(1+\sup_{0 \le t \le T}\|Z_t\|^2_{\VV}\right) \left(\|Y_{T/2}\|^2_{\HH}+\int_{T/2}^T (1+\|Z_s\|^4_{\VV})\dif s\right)^{\frac 12} \\
\le &  C\left(1 +\sup_{0 \le t \le T} \|Z_t\|^4_{\VV}\right).
\end{align*}
Hence, by Lemma \ref{lem Z} (taking $\gamma=1$ there), we obtain that for any $p \in (0, (1+\delta)\alpha/8)$,
\begin{align*}
\E^{x}\left[\|Y_{T}\|^p_{\delta}\right] \le & C\left(1+\E^x\left[\sup_{0 \le t \le T} \|Z_t\|^{4p}_{\VV}\right]\right )\\
\le& C_{\alpha, p} T^{\frac {p}{\alpha}}\left(1+T^{\frac{1-\theta}{2}}\right)\le  2C_{\alpha, p} T,
\end{align*}
where $C_{\delta, p}$ is independent of $x$ and $T$.

The proof is complete.
\end{proof}

\vskip0.3cm

 By Lemma \ref{lem Z} and Lemma \ref{l:Y}, we can get the following estimate.

\begin{thm}\label{l: X}
For all $T>0$,   $\delta \in (0,1)$ and $p \in (0,(1+\delta)\alpha/8)$, we have
$$
\E^{x}\left[\|X_{T}\|^p_{\delta}\right] \le C_{\delta,p}T,
$$
where the constant $C_{\delta,p}$  does not depend on the initial  value $X_0=x$ and $T$.
 Consequently,the Markov property, it follows from the Markov property that
 $$
\sup_{t\ge T}\E^{x}\left[\|X_{t}\|^p_{\delta}\right]\le C_{\delta,p}T.
$$
\end{thm}

\subsubsection{The hyper-exponential Recurrence}

In this part, we will verify the hyper-exponential recurrence condition \eqref{condition 2}.
\vskip0.3cm

For any $\delta\in(0,1), M>0$, define the hitting time  of $\{X_n\}_{n\ge1}$:
\begin{equation}
\tau_M:=\inf\{k\ge1:\|X_{k}\|_{{\delta}}\le M\}.
\end{equation}
 Let
 $$
 K:=\{x\in \HH_{\delta}: \|x\|_{{\delta}}\le M\}.
 $$
 Clearly, $K$ is compact in $\HH$. Recall the definitions of  $\tau_K$ and $\tau_K^{(1)}$ in \eqref{stopping time}. It is obvious  that
\begin{equation}
\tau_K\le \tau_M,\ \ \ \ \ \tau_K^{(1)}\le \tau_M.
\end{equation}
This fact, together with the following important theorem,  implies the hyper-exponential recurrence condition  \eqref{condition 2}.

\begin{theorem}\label{thm hyper-exp 2} For any $\lambda>0$,  there exists $M=M_{\lambda, \delta}$  such that
$$
\sup_{\nu\in\mathcal M_1(\HH)}\E^{\nu}[e^{\lambda\tau_M}]<\infty.
$$
\end{theorem}

\begin{proof}
For any $n\in\N$, let
$$
B_n:=\left\{\|X_{j}\|_{{\delta}}>M; j=1,\cdots,n\right\}=\{\tau_M>n\}.
$$
By the Markov property of $\{X_n \}_{n\in\N}$, Chebychev's inequality and Theorem \ref{l: X}, we obtain that for any $\nu\in \M_1(\HH)$, $p \in (0,\alpha/4)$,
\begin{align*}
\PP_{\nu}(B_n)=&\PP_{\nu}(B_{n-1})\cdot\PP_{\nu}(B_n|B_{n-1})\\
\le&\PP_{\nu}(B_{n-1})\cdot \frac{\E_{X_{n-1}}\left[\|X_n\|_{{\delta}}^p\right]}{M^p}\\
\le & \PP_{\nu}(B_{n-1})\cdot\frac{C_{\delta,p}}{M^p},
\end{align*}
where $C_{\delta,p}$ is the constant in Lemma \ref{l: X} (taking $T=1$).

By induction, we have for any $n\ge0$,
$$
\PP_{\nu}(\tau_M>n)=\PP_{\nu}(B_n)\le \left(\frac{C_{\delta,p}}{M^p}\right)^n.
$$
This inequality, together with Fubini's theorem, implies that for any $\lambda>0,\nu\in\M_1(\HH)$,
\begin{align*}
\E_{\nu}\left[ e^{\lambda\tau_M}\right]=&\int_0^{\infty}\lambda e^{\lambda t}\PP_{\nu}(\tau_M>t)\dif t\\
\le &\sum_{n=0}^{\infty}\lambda e^{\lambda (n+1)}\PP_{\nu}(\tau_M>n)\\
\le &\sum_{n=0}^{\infty}\lambda e^{\lambda (n+1)}\left(\frac{C_{\delta,p}}{M^p}\right)^n,
\end{align*}
which is finite as $M>(C_{\delta,p} e^{\lambda})^{1/p}$.

The proof is complete.
\end{proof}

\vskip0.3cm
\noindent{\bf Acknowledgments}: The authors  would like to gratefully thank Feng-Yu Wang for some very useful discussions. R. Wang thanks the  Faculty of Science and Technology, University of Macau, for finance support and hospitality.   He is supported by NNSFC(11301498, 11431014, 11671076).  L. Xu is supported by the grants: NNSFC(11571390), MYRG2015-00021-FST and Science and Technology Development Fund, Macao S.A.R FDCT 030/2016/A1.

\bibliographystyle{amsplain}

\begin{thebibliography}{99}


\bibitem{Ap09}Applebaum D. (2009) \textit{L\'evy processes and stochastic calculus}. Second edition. Cambridge Studies in Advance Mathematics, {\bf116},  Cambridge University Press.



\bibitem{DPZ96}
  Da Prato G.  and Zabczyk J. (1996)  \textit{Ergodicity for infinite-dimensional systems}. London  Mathematical Society Lecture Note Series,  \textbf{229}, Cambridge University  Press, Cambridge.


\bibitem{DZ} Dembo A. and Zeitouni O. (1998) \textit{ Large deviations techniques and applications}. Second
edition, Applications of Mathematics, {\bf38}, Springer-Verlag.



\bibitem{DS} Deuschel J.D. and Stroock D. (1989) \textit{Large deviations}.  Pure and Applied Mathematicas \textbf{137}. Academic Press, Inc., Boston, MA.

\bibitem{Do08} Dong Z. (2008) On the uniqueness of invariant measure of the Burgers equation  driven by L\'evy processes. \textit{J. Theoret. Probab.}, \textbf{21},  322-335.

\bibitem{DXi11}Dong Z. and Xie Y. (2011) Ergodicity of stochastic 2{D}  {N}avier-{S}tokes equations with {L}\'evy noise.  \textit{J.  Differential Equations},  \textbf{251},  196-222.


\bibitem{DXZ14JSP}Dong Z., Xu L. and  Zhang X. (2014)  Exponential ergodicity of stochastic Burgers equations driven by $\alpha$-stable processes.  \textit{J. Stat. Phys.}, \textbf{154}(4), 929-49.

\bibitem{DXZ09}Dong Z., Xu T. and  Zhang T. (2009) Invariant measures for  stochastic evolution equations of pure jump type. \textit{Stochastic Process. Appl.},  \textbf{119},  410-427.


\bibitem{DV} Donsker M.D.  and   Varadhan S.R.S.,   Asmptotic evaluation of certain Markov process expectations for large time, I-IV,
{\it Comm. Pure Appl. Math.} \textbf{28}, 1-47  (1975); \textbf{28}, 279-301  (1975); \textbf{29}, 389-461  (1976); \textbf{36}, 183-212  (1983).


\bibitem{Doob} Doob J. L. (1948) Asymptotic properites of Markov transition probability. \textit{Trans. Am. Math. Soc.}, \textbf{64}, 393-421.





\bibitem{FuXi09} Fernando P. W., Hausenblas E. and Razafimandimby P.A. (2016) Irreducibility and exponential mixing of some stochastic hydrodynamical systems driven by pure jump noise. \textit{Comm. Math. Phys.} 348, no. 2, 535-565.

\bibitem{FuXi09}
 Funaki T. and Xie B. (2009) A stochastic heat equation with
the distributions of {L}\'evy processes as its invariant measures.
\textit{Stochastic Process. Appl.}, \textbf{119},  307-326.



\bibitem{Gou1} Gourcy M. (2007)  Large deviation principle of occupation measure for stochastic burgers equaiton. {\it Ann. Inst. H. Poincar\'{e}},   \textbf{43},  441-459.


\bibitem{Gou2} Gourcy M. (2007)  A large deviation principle for 2D stochastic Navier-Stokes equation. {\it Stochastic Process. Appl.}, \textbf{117}, 904-927.

\bibitem{IW} Ikeda N. and Watanabe S. (1981) \textit{Stochastic Differential Equations and Diffusion Processes}. North-Holland Publishing Co., Amsterdam.



\bibitem{JNPS1} Jak\u{s}i\`c V., Nersesyan V., Pillet C.  and Shirikyan A.,  (2015)  Large deviaitons from a stationary measure for a class of dissipative PDE's with random kicks.   {\it Comm. Pure Appl. Math.}, \textbf{12}, 2108-2143.

\bibitem{JNPS2} Jak\u{s}i\`c V., Nersesyan V., Pillet C.  and Shirikyan A. (2015) Large deviaitons and Gallavotti-Cohen principle for dissipative PDE's with rough noise.  {\it Comm.  Math. Phys.,} \textbf{336}(1), 131-170.

\bibitem{Kif} Kifer Y. (1980) Large deviations in dynamical systems and stochastic processes. {\it Trans. Amer. Math. Soc.}, \textbf{321}, 505-524.



\bibitem{Mas} Masuda H. (2007) Ergodicity and exponential $\beta$-mixing bounds for multidimensional diffusions with jumps. \textit{Stochastic Process. Appl.}, \textbf{117},  35-56.


\bibitem{PeZa07} Peszat S.  and  Zabczyk J. (2007) \textit{Stochastic partial differential equations with
 {L}\'evy noise. An evolution equation approach.} Encyclopedia of Mathematics and its Applications,  \textbf{113},  Cambridge University Press, Cambridge.


\bibitem{PSXZ11}
Priola E., Shirikyan A.,~Xu L. and Zabczyk J. (2012) Exponential ergodicity and regularity for equations with L\'evy noise.
 \textit{Stochastic Process. Appl.}, \textbf{122}, 106-133.



\bibitem{PZ11}
Priola E. and Zabczyk J. (2011) Structural properties of semilinear
  {SPDE}s driven by cylindrical stable processes. \textit{Probab. Theory Related
  Fields}, \textbf{149},  97-137.




 \bibitem{RY} Revuz D. and Yor M. (1999),  \textit{Continuous martingales and Brownian motion}. 3rd ed., Springer, Berlin.



\bibitem{sato} Sato K. (1999),  \textit{L\'evy processes and infinite divisible distributions.}
 Cambridge University Press, Cambridge.


\bibitem{Wa} Wang F. Y. (2016),  Integration by parts formula and applications for SPDEs with jumps. \textit{Stochastics} 88, no. 5, 737-750.

\bibitem{WXX}  Wang R., Xiong J. and Xu L. (2016) Irreducibility of stochastic real Ginzburg-Landau equation driven by $\alpha$-stable noises and applications, To appear in \textit{ Bernoulli}.



\bibitem{WXX2}  Wang R., Xiong J. and Xu L. (2016) Large deviation principle of occupation measures for non-linear monotone SPDEs. arXiv:1601.06270v1.


\bibitem{Wu01} Wu L. (2001) Large and moderate deviations and exponential convergence for stochastic damping Hamiltionian systems. \textit{Stochastic Process. Appl.}, \textbf{91}, 205-238.






\bibitem{Xu13} Xu L. (2013) Ergodicity of the stochastic real Ginzburg-Landau equation driven by $\alpha$-stable noises. \textit{Stochastic Process. Appl.}, \textbf{123}, 3710-3736.



\bibitem{Zh}  Zhang X. (2013) Derivative formula and gradient estimate for SDEs driven by $\alpha$-stable processes,
 \emph{Stoch. Process. Appl.}, \textbf{123},  1213-1228.




\end{thebibliography}

\end{document}